\documentclass[11pt,reqno]{amsart}

\usepackage{amssymb,amsmath,amsthm}
\usepackage{a4wide,comment,enumitem,url}
\usepackage{xcolor}
\usepackage{bm}
\usepackage{cleveref}
\newtheorem{theorem}{Theorem}
\newtheorem{lemma}[theorem]{Lemma}

\theoremstyle{remark}

\newtheorem*{remarks}{Remarks}

\numberwithin{theorem}{section} \numberwithin{equation}{section}

\setlist[enumerate]{leftmargin=*,label=\rm{(\arabic*)}}
\setlist[itemize]{leftmargin=*}

\newcommand{\CE}{\mathcal{E}}
\newcommand{\CI}{\mathcal{I}}
\newcommand{\CJ}{\mathcal{J}}
\newcommand{\CK}{\mathcal{K}}
\newcommand{\CS}{\mathcal{S}}

\renewcommand{\b}{\beta}
\newcommand{\e}{\varepsilon}
\renewcommand{\k}{\kappa}

\renewcommand{\t}{\tau}
\newcommand{\w}{\omega}
\newcommand{\Ga}{\Gamma}

\newcommand{\ssum}{{\textstyle{\displaystyle\sum}}}

\newcommand{\flo}[1]{\lfloor #1\rfloor}

\newcommand{\R}{\mathbb{R}}
\newcommand{\C}{\mathbb{C}}

\newcommand{\K}{\mathbb{K}}

\newcommand{\Z}{\mathbb{Z}}
\newcommand{\N}{\mathbb{N}}

\renewcommand{\a}{\alpha}

\newcommand{\vth}{\vartheta}

\newcommand{\re}{\textnormal{Re}}

\renewcommand{\pmod}[1]{\ \, \left( \mathrm{mod} \, #1 \right)}
\newcommand{\Pmod}[1]{\ \, ( \mathrm{mod} \, #1 )}

\makeatletter
\@namedef{subjclassname@2020}{%
	\textup{2020} Mathematics Subject Classification}
\makeatother

\allowdisplaybreaks

\title{A Rademacher-type exact formula for partitions without sequences}
\author{Walter Bridges}
\author{Kathrin Bringmann}
\address{University of Cologne, Department of Mathematics and Computer Science, Weyertal 86-90, 50931 Cologne, Germany }
\email{wbridges@uni-koeln.de}
\email{kbringma@math.uni-koeln.de}


\begin{document}
\maketitle

\begin{abstract}
	In this paper we prove an exact formula for the number of partitions without sequences. By work of Andrews, the corresponding generating function is a product of a modular form and a mock theta function, giving an overall weight of 0. The proof requires evaluating and bounding Kloosterman sums and the Circle Method.
\end{abstract}

\section{Introduction and statement of results}

A {\it partition} of a non-negative integer $n$ is a decomposition into the sum of non-increasing non-negative integers. Denote the number of partitions of $n$ by $p(n)$. The corresponding generating function is
\begin{equation}\label{E:pQop}
	P(q):=\sum_{n=0}^\infty p(n)q^n = \prod_{n=0}^\infty \frac{1}{1-q^n} = \frac{1}{(q;q)_\infty},
\end{equation}
where, for $a\in\C$ and $n\in\N_0\cup\{\infty\}$, we define $(a)_n=(a;q)_n:=\prod_{j=0}^{n-1}(1-aq^j)$.

An important question in the theory of partitions is to determine exact formulas or asymptotics for functions such as $p(n)$. Note that the generating function in \eqref{E:pQop} is (essentially) a modular form. A Tauberian Theorem \cite{HR} shows that the following asymptotic holds
\begin{equation}\label{E:partas}
	p(n) \sim \frac{1}{4n\sqrt3}e^{\pi\sqrt\frac{2n}{3}} \qquad (n\to\infty).
\end{equation}

Building on Hardy and Ramanujan's earlier developments, Rademacher and Zuckerman later proved exact asymptotic series expansions for functions like $p(n)$ using the Circle Method \cite{RadZ}. To state this, define the {\it Kloostermann sums}\footnote{Note that in particular $\w_{h,k}$ only depends on $h\Pmod k$.}, with $\w_{h,k}$ given in \eqref{E:omega},
\[
	A_k(n) := \sum_{h\Pmod k^*} \w_{h,k}e^{-\frac{2\pi inh}{k}}.
\]
The $*$ indicates that $h$ only runs over those elements$\Pmod k$ that are coprime to $k$. Then
\begin{equation}\label{E:partexact}
	p(n) = \frac{2\pi}{(24n-1)^\frac34}\sum_{k\ge1} \frac{A_k(n)}{k}I_\frac32\left(\frac{\pi\sqrt{24n-1}}{6k}\right),
\end{equation}
where $I_\k$ denotes the Bessel function of order $\k$. Using that $I_\k(x)\sim\frac{e^x}{\sqrt{2\pi x}}$ (as $x\to\infty$), one recovers \eqref{E:partas}.

Another important example in the study of coefficients of hypergeometric series and automorphic forms is Ramanujan's third order mock theta function
\begin{equation}\label{fq}
	f(q) = \sum_{n=0}^\infty \a(n)q^n := 1 + \sum_{n=1}^\infty \frac{q^{n^2}}{(-q;q)_n^2}.
\end{equation}
The mock theta functions were introduced by Ramanujan in his last letter to Hardy \cite{Ram}. By work of Zwegers \cite{Z}, $f(q)$ is ``the holomorphic part of a harmonic Maass form''. This in particular implies that $f$ is not quite modular, but its modular transformations involve ``Mordell-type integrals''. These integrals were absorbed into the error terms of the asymptotic series expansion for $\a(n)$ obtained by Dragonette \cite{D} and Andrews \cite{A66}. They then conjectured an exact formula for $\a(n)$. Work of the second author and Ono used so-called Maass--Poincar\'e series to solve this conjecture \cite{BO}. The exact formula closely resembles \eqref{E:partexact} meaning that the non-modularity of $f(q)$ does not influence the shape of the exact formula.

The situation drastically changes if one multiplies mock theta functions with modular forms. These functions arise in many settings, for example in the study of probability, mathematical physics, and partition theory. One particular example gives partitions that do not contain any consecutive integers as parts. Such partitions were first explored by MacMahon \cite{Mac} and arise in connection with certain probability models as well as in the study of threshold growth in cellular automata \cite{HLR} (also see \cite{AEPR}). Let $p_2(n)$ denote the number of such partitions of size $n$. Andrews \cite{A05} proved that we have the generating function
\begin{equation*}
	G_2(q) := \sum_{n=0}^\infty p_2(n)q^n = \frac{\left(-q^3;q^3\right)_\infty}{\left(q^2;q^2\right)_\infty}\chi(q),
\end{equation*}
with the third order mock theta function
\[
	\chi(q) := \sum_{n=0}^\infty \frac{(-q;q)_nq^{n^2}}{\left(-q^3;q^3\right)_n}.
\]
The combinatorics of the generating function was further explored in \cite{BLM}. From work of Holroyd, Liggett, and Romik \cite{HLR} it follows that
\begin{equation*}
	\frac1ne^{\frac{2\pi}{3}\sqrt n} \ll p_2(n) \ll \frac{1}{\sqrt n}e^{\frac{2\pi}{3}\sqrt n}.
\end{equation*}
The second author and Mahlburg \cite{BM} strongly improved upon this to find an asymptotic series for $p_2(n)$ with an error term of size $\log(n)$. In this paper we strengthen this result and prove an exact formula. For this, we require some notation. For $b\in\R$, $k\in\N$, and $\nu\in\Z$ define
\[
	\CI_{b,k,\nu}(n) := \int_{-1}^1 \frac{\sqrt{1-x^2}I_1\left(\frac{2\pi}{k}\sqrt{2bn\left(1-x^2\right)}\right)}{\cosh\left(\frac{\pi i}{k}\left(\nu-\frac16\right)-\frac\pi k\sqrt\frac b3x\right)} dx.
\]
Moreover, let the Kloosterman sums $K_k^{[4]}$, $K_k^{[6]}$, $K_k^{[8]}$, and $\CK_k$ be given in \eqref{E:K4}, \eqref{E:K6}, \eqref{E:K8}, and \eqref{E:K}, respectively.\footnote{Note that these Kloosterman sums may also be written in terms of classical Kloosterman sums -- see \Cref{S:KSums} for details.}

\begin{theorem}\label{T:p2}
	We have, for $n\in\N$,
	\begin{multline*}
		p_2(n) = \frac{5\pi}{36\sqrt{6n}}\sum_{\substack{k\ge1\\\gcd(k,6)=2}} \frac{1}{k^2}\sum_{\nu\Pmod k} (-1)^\nu K_k^{[4]}(\nu;n)\CI_{\frac{5}{36},k,\nu}(n)\\
		+ \frac{\pi}{6\sqrt{6n}}\sum_{\substack{k\ge1\\\gcd(k,6)=3}} \frac{1}{k^2}\sum_{\nu\Pmod k} (-1)^\nu K_k^{[6]}(\nu;n)\CI_{\frac16,k,\nu}(n) + \frac{\pi}{6\sqrt n}\sum_{\substack{k\ge1\\\gcd(k,6)=1}} \frac{\CK_k(n)}{k^2}I_1\left(\frac{2\pi\sqrt n}{3k}\right)\\
		+ \frac{\pi}{18\sqrt{6n}}\sum_{\substack{k\ge1\\\gcd(k,6)=1}} \frac{1}{k^2}\sum_{\nu\Pmod k} (-1)^\nu K_k^{[8]}(\nu;n)\CI_{\frac{1}{18},k,\nu}(n).
	\end{multline*}
\end{theorem}

\begin{remarks}\leavevmode\newline
	(1) \Cref{T:p2} is the first case of an exact formula for a mixed mock modular form of weight $0$. In \cite{BMa} the much simpler case of a mixed modular form of weight $-\frac32$ was considered. It would be interesting to develop methods to prove exact formulas in the case of weight $\frac{1}{2}$, for example Rhoades proved an asymptotic series for the number of strongly unimodal sequences of size $n$ \cite{Rhoades}. This is likely a difficult problem, as unlike for $f(q)$ there is no theory of Poincar\'{e} series for mixed mock modular forms.\\
	(2) As in \cite{BM}, one can use \Cref{T:p2} to determine the first term in the asymptotic expansion of $p_2(n)$.\\
	(3) The unrestricted partition function is log-concave (see \cite{DP}), and we conjecture that $p_2(n)^2-p_2(n+1)p_2(n-1) \geq 0$ for $n \geq 482$ and all even $n\geq 2$.
\end{remarks}

The paper is organized as follows. In \Cref{S:pre} we recall some basic facts on multipliers, transformation laws, and bounds for Mordell-type integrals. \Cref{S:KSums} is devoted to Kloosterman sums (rewriting and bounding). In \Cref{S:CM} we apply the Circle Method to prove \Cref{T:p2}.

\section*{Acknowledgements}
The authors thank Lukas Mauth for providing helpful feedback on an earlier draft. This project has received funding from the European Research Council (ERC) under the European Union's Horizon 2020 research and innovation programme (grant agreement No. 101001179).

\section{Preliminaries}\label{S:pre}

\subsection{Modularity of the partition function}

We require the modularity of the partition function $P(q)$. We state these in terms of the Circle Method notation. Let $q:=e^{\frac{2\pi i}{k}(h+iz)}$, $q_1:=e^{\frac{2\pi i}{k}(h'+\frac iz)}$ with $z\in\C$ with $\re(z)>0$, $h,k\in\N$,\footnote{We often impose extra conditions on $h'$; whenever we do so, we mention these.} $hh'\equiv-1\Pmod k$. Moreover let $\w_{h,k}$ be defined through
\[
	P(q) = \w_{h,k}z^\frac12e^{\frac{\pi}{12k}\left(z^{-1}-z\right)}P(q_1).
\]
Then we have (see \cite[equation (5.2.4)]{A80})
\begin{equation}\label{E:omega}
	\w_{h,k} :=
	\begin{cases}
		\left(\frac{-k}{h}\right)e^{-\pi i\left(\frac14(2-hk-h)+\frac{1}{12}\left(k-\frac1k\right)\left(2h-h'+h^2h'\right)\right)} & \text{if }h\text{ is odd},\\
		\left(\frac{-h}{k}\right)e^{-\pi i\left(\frac14(k-1)+\frac{1}{12}\left(k-\frac1k\right)\left(2h-h'+h^2h'\right)\right)} & \text{if }k\text{ is odd}.
	\end{cases}
\end{equation}
Here $(\frac\cdot\cdot)$ denotes the Kronecker symbol.

\subsection{A splitting of $G_2(q)$}

Employing a mock theta-function identity of Ramanujan, Andrews \cite{A05} showed the following decomposition
\begin{equation}\label{E:G2f}
	G_2(q) = \frac{\left(q^6;q^6\right)_\infty}{4\left(q^2;q^2\right)_\infty\left(q^3;q^3\right)_\infty}f(q) + \frac{3\left(q^3;q^3\right)^3_\infty}{4(q;q)_\infty\left(q^2;q^2\right)_\infty\left(q^6;q^6\right)_\infty},
\end{equation}
where $f(q)$ is defined in \eqref{fq}. We denote the two terms on the right-side of equation \eqref{E:G2f} by
\[
	g_1(q) := \frac{\left(q^6;q^6\right)_\infty}{4\left(q^2;q^2\right)_\infty\left(q^3;q^3\right)_\infty}f(q) =: \sum_{n=0}^\infty a(n)q^n,\qquad g_2(q) := \frac{3\left(q^3;q^3\right)^3_\infty}{4(q;q)_\infty\left(q^2;q^2\right)_\infty\left(q^6;q^6\right)_\infty}.
\]
Moreover we require
\[
	\xi(q) := \frac{\left(-q^3;q^3\right)_\infty}{\left(q^2;q^2\right)_\infty} =: \sum_{n=0}^\infty r(n)q^n.
\]

\subsection{Mordell integrals}

Throughout, we let $0\le h<k\le N$ with $\gcd(h,k)=1$, and $z=k(N^{-2}-i\Phi)$ with $-\vth_{h,k}'\le\Phi\le\vth_{h,k}''$. Here
\[
	\vth_{h,k}' := \frac{1}{k(k_1+k)},\qquad \vth_{h,k}'' := \frac{1}{k(k_2+k)},
\]
where $\frac{h_1}{k_1}<\frac hk<\frac{h_2}{k_2}$ are adjacent Farey fractions in the Farey sequence of order $N\in\N$. Below, we let $N\to\infty$.

The following Mordell-type integral occurs in the modular transformation laws of the mock theta function $f(q)$
\begin{equation*}
	I_{k,\nu}(z) := \int_\R \frac{e^{-\frac{3\pi zx^2}{k}}}{\cosh\left(\frac{\pi i}{k}\left(\nu-\frac16\right)-\frac{\pi zx}{k}\right)} dx.
\end{equation*}
In Lemma 3.1 of \cite{BM} these integrals were approximated. To state this result, let $\CJ_{b,k,\nu}(z):=ze^\frac{\pi b}{kz}I_{k,\nu}(z)$, and define the principal part truncation of $\CJ_{b,k,\nu}$ as
\[
	\CJ_{b,k,\nu}^*(z) := \sqrt\frac b3\int_{-1}^1 \frac{e^{\frac{\pi b}{kz}\left(1-x^2\right)}}{\cosh\left(\frac{\pi i}{k}\left(\nu-\frac16\right)-\frac\pi k\sqrt\frac b3x\right)} dx.
\]

\begin{lemma}\label{L:Jb}
	If $b\in\R$ and $\nu\in\Z$ with $0<\nu\le k$, then we have the following, as $z\to0$:
	\begin{enumerate}
		\item\label{I:Jb1} If $b\le0$, then we have
		\[
			|\CJ_{b,k,\nu}(z)| \ll \frac{1}{\left|\frac\pi2-\frac\pi k\left(\nu-\frac16\right)\right|}.
		\]
		
		\item\label{I:Jb2} If $b>0$, then $\CJ_{b,k,\nu}(z)=\CJ_{b,k,\nu}^*(z)+\CE_{b,k,\nu}(z)$, where the error term satisfies for $0<\nu\le k$
		\[
			|\CE_{b,k,\nu}(z)| \ll \frac{1}{\left|\frac\pi2-\frac\pi k\left(\nu-\frac16\right)\right|}.
		\]
	\end{enumerate}
\end{lemma}

\section{Kloosterman sums}\label{S:KSums}

\subsection{Rademacher's Kloosterman sums}

Recall that $k_1=k_1(h,k)$ is the denominator of the fraction preceeding $\frac hk$ in the Farey sequence of order $N\in\N$. By \cite{Rad} we have the following bounds for Kloosterman sums.

\begin{lemma}\label{L:Klooster}
	We have, for $k\in\N$, $n,m,\ell\in\Z$, $n\ne0$ with $N+1\le\ell\le N+k+1$, for $\e>0$,
	\begin{align*}
		K_k(n,m) := \sum_{h\Pmod k^*} e^{-\frac{2\pi i}{k}\left(nh-mh'\right)} &= O_\e\left(k^{\frac23+\e}\gcd(|n|,k)^\frac13\right),\\ 
		\K_{k,\ell}(n,m) := \sum_{\substack{h\Pmod k^*\\N<k+k_1\le\ell}} e^{-\frac{2\pi i}{k}\left(nh-mh'\right)} &= O_\e\left(k^{\frac23+\e}\gcd(|n|,k)^\frac13\right).
	\end{align*}
\end{lemma}

We now investigate certain Kloosterman sums that occur when using the Circle Method for $g_1(q)$. For this, we distinguish cases based on $\gcd(k,6)$.

\subsection{$6\mid k$}

Define the following Kloosterman sums ($n,m\in\N_0$)
\begin{align*}
	K_k^{[1]}(n,m) &:= \sum_{\substack{0\le h<k\\\gcd(h,k)=1}} \frac{\w_{h,k}\w_{h,\frac k2}\w_{h,\frac k3}}{\w_{h,\frac k6}}e^{\frac{\pi i}{2}\left(1-\frac{3k}{2}\right)h'}e^{\frac{2\pi i}{k}\left(-nh+mh'\right)},\\
	K_k^{[2]}(\nu;n,m) &:= \sum_{\substack{0\le h<k\\\gcd(h,k)=1}} \frac{\w_{h,k}\w_{h,\frac k2}\w_{h,\frac k3}}{\w_{h,\frac k6}}e^{\frac{\pi i}{k}\left(-3\nu^2+\nu\right)h'}e^{\frac{2\pi i}{k}\left(-nh+mh'\right)}.
\end{align*}
Note that these Kloosterman sums only depend on the residue class of $h,h'\Pmod k$. We also need the incomplete versions, $N+1\le\ell\le N+k-1$,
\begin{align*}
	\K_{k,\ell}^{[1]}(n,m) &:= \sum_{\substack{h\Pmod k^*\\N<k+k_1\le\ell}} \frac{\w_{h,k}\w_{h,\frac k2}\w_{h,\frac k3}}{\w_{h,\frac k6}}e^{\frac{\pi i}{2}\left(1-\frac{3k}{2}\right)h'}e^{\frac{2\pi i}{k}\left(-nh+mh'\right)},\\
	\K_{k,\ell}^{[2]}(\nu;n,m) &:= \sum_{\substack{h\Pmod k^*\\N<k+k_1\le\ell}} \frac{\w_{h,k}\w_{h,\frac k2}\w_{h,\frac k3}}{\w_{h,\frac k6}}e^{\frac{\pi i}{k}\left(-3\nu^2+\nu\right)h'}e^{\frac{2\pi i}{k}\left(-nh+mh'\right)}.
\end{align*}

The following lemma rewrites the multiplier.

\begin{lemma}\label{L:mult1}
	We have
	\[
		\frac{\w_{h,k}\w_{h,\frac k2}\w_{h,\frac k3}}{\w_{h,\frac k6}} = -e^{\pi i\frac{5k+18}{36}h}.
	\]
\end{lemma}

\begin{proof}
	We use \eqref{E:omega}. Since $2\mid k$, we have that $h$ is odd. A direct calculation gives that
	\[
		\frac{\w_{h,k}\w_{h,\frac k2}\w_{h,\frac k3}}{\w_{h,\frac k6}} = -e^{2\pi iA},
	\]
	with
	\[
		A := -\frac{5h^2h'k}{72} + \frac{5k+18}{72}h + \frac{5h'k}{72} = \frac{5\frac k6\left(1-h^2\right)h'}{12} + \frac{5k+18}{72}h.
	\]
	Now $\gcd(h,6)=1$. Thus $\frac{1-h^2}{12}\in\Z$. This gives the claim.
\end{proof}

We next bound the Kloosterman sums.

\begin{lemma}\label{L:Kkmn}
	We have, for $\e>0$,
	\begin{align*}
		K_k^{[1]}(n,m),\ \K_{k,\ell}^{[1]}(n,m),\ K_k^{[2]}(\nu;n,m),\ \K_{k,\ell}^{[2]}(\nu;n,m) \ll_\e n^\frac13k^{\frac23+\e}.
	\end{align*}
\end{lemma}

\begin{proof}
	Using \Cref{L:mult1} and \Cref{L:Klooster}, we have
	\begin{align*}
		K_k^{[1]}(n,m) &= K_k\left(n-(5k+18)\frac{k}{72},m+\frac k4\left(1-\frac{3k}{2}\right)\right)\\
		&\ll_\e k^{\frac23+\e}\gcd\left(\left|n-(5k+18)\frac{k}{72}\right|,k\right)^\frac13 \ll k^{\frac23+\e}n^\frac13.
	\end{align*}
	
	The remaining cases may be proved analogously.
\end{proof}

\subsection{$\gcd(k,6)=2$}

Define the Kloosterman sums
\begin{align*}
	K_k^{[3]}(n,m) &:= \sum_{\substack{0\le h<k\\\gcd(h,k)=1\\3\mid h'}} \frac{\w_{h,k}\w_{h,\frac k2}\w_{3h,k}}{\w_{3h,\frac k2}}e^{\frac{\pi i}{2}\left(1-\frac{3k}{2}\right)h'}e^{\frac{2\pi i}{k}\left(-nh+\frac{mh'}{3}\right)},\\
	K_k^{[4]}(\nu;n,m) &:= \sum_{\substack{0\le h<k\\\gcd(h,k)=1\\3\mid h'}} \frac{\w_{h,k}\w_{h,\frac k2}\w_{3h,k}}{\w_{3h,\frac k2}}e^{\frac{\pi i}{k}\left(-3\nu^2+\nu\right)h'}e^{\frac{2\pi i}{k}\left(-nh+\frac{mh'}{3}\right)}.
\end{align*}
For the range in the sums above, we may again choose any representatives $h,h' \pmod{k}$ satisfying $3\mid h'$. Note that the condition $3\mid h'$ enforces that we need to change $h'\mapsto h'+3k$. We also use the abbreviation
\begin{equation}\label{E:K4}
	K_k^{[4]}(\nu;n) := K_k^{[4]}(\nu;n,0).
\end{equation}

Again, we require the incomplete versions, $N+1\le\ell\le N+k-1$,
\begin{align*}
	\K_{k,\ell}^{[3]}(n,m) &:= \sum_{\substack{h\Pmod k^*\\N<k+k_1\le\ell\\3\mid h'}} \frac{\w_{h,k}\w_{h,\frac k2}\w_{3h,k}}{\w_{3h,\frac k2}}e^{\frac{\pi i}{2}\left(1-\frac{3k}{2}\right)h'}e^{\frac{2\pi i}{k}\left(-nh+\frac{mh'}{3}\right)},\\
	\K_{k,\ell}^{[4]}(\nu;n,m) &:= \sum_{\substack{h\Pmod k^*\\N<k+k_1\le\ell\\3\mid h'}} \frac{\w_{h,k}\w_{h,\frac k2}\w_{3h,k}}{\w_{3h,\frac k2}}e^{\frac{\pi i}{k}\left(-3\nu^2+\nu\right)h'}e^{\frac{2\pi i}{k}\left(-nh+\frac{mh'}{3}\right)}.
\end{align*}

Again, we evaluate the multiplier.

\begin{lemma}\label{L:we}
	We have 
	\[
		\frac{\w_{h,k}\w_{h,\frac k2}\w_{3h,k}}{\w_{3h,\frac k2}} = e^{\frac{2\pi i}{k}\left(\frac{k(k+2)}{8}h-\frac{k^2+2}{18}h'\right)}.
	\]
\end{lemma}

\begin{proof}
	Note that $2\mid k$ implies that $h$ is odd, so we compute, using \eqref{E:omega},
	\[
		\frac{\w_{h,k}\w_{h,\frac k2}\w_{3h,k}}{\w_{3h,\frac k2}} = -e^{2\pi iB},
	\]
	where
	\[
		B := \frac{1}{72k}\left(\left(-9h^2h'+5h'+9h\right)k^2+(18h-36)k-8h'\right).
	\]
	As $2\mid k$, we have $2\nmid h$ and thus $h^2\equiv1\Pmod4$. Therefore, we obtain
	\[
		B \equiv \left(\frac k8+\frac14\right)h + \left(-\frac k8+\frac{5k}{72}-\frac{1}{9k}\right)h' - \frac12 \equiv \frac{k+2}{8}h + \left(-\frac{k}{18}-\frac{1}{9k}\right)h' + \frac12 \Pmod1.
	\]
	This gives the claim.
\end{proof}

Again, we need bounds for these Kloosterman sums.

\begin{lemma}\label{L:Kkl34}
	We have, for $\e>0$,
	\[
		K_k^{[3]}(n,m),\ \K_{k,\ell}^{[3]}(n,m),\ K_k^{[4]}(\nu;n,m),\ \K_{k,\ell}^{[4]}(\nu;n,m) \ll_\e n^\frac13k^{\frac23+\e}.
	\]
\end{lemma}

\begin{proof}
	Note that $\frac{h'}{3}\equiv[3]_kh'\Pmod k$, where $[a]_b$ denotes the inverse of $a\Pmod b$. Thus we have, using \Cref{L:we},
	\[
		K_k^{[3]}(n,m) = \sum_{\substack{h\Pmod k^*\\3\mid h'}} e^{\frac{2\pi i}{k}\left(\left(\frac{k(k+2)}{8}-n\right)h+\left(-\frac{k^2+2}{18}+\frac k4\left(1-\frac{3k}{2}\right)+[3]_km\right)h'\right)}.
	\]
	To get rid of the condition $3\mid h'$, we change $h'\mapsto3h'$ and $h\mapsto[3]_kh$ to obtain that
	\[
		K_k^{[3]}(n,m) = K_k\left([3]_k\left(n-\frac{k(k+2)}{8}\right),m-\frac{k^2+2}{6}+\frac{3k}{4}\left(1-\frac{3k}{2}\right)\right).
	\]
	Using \Cref{L:Klooster} gives the claim. The remaining cases are treated in the same way.
\end{proof}

\subsection{$\gcd(k,6)=3$}

We require the following Kloosterman sums
\begin{align*}
	K_k^{[5]}(n,m) &:= \sum_{\substack{0\le h<k\\\gcd(h,k)=1\\8\mid h'}} \frac{\w_{h,k}\w_{2h,k}\w_{h,\frac k3}}{\w_{2h,\frac k3}}e^\frac{3\pi ih'}{4k}e^{\frac{2\pi i}{k}\left(-nh+\frac{mh'}{2}\right)},\\
	K_k^{[6]}(\nu;n,m) &:= \sum_{\substack{0\le h<k\\\gcd(h,k)=1\\8\mid h'}} \frac{\w_{h,k}\w_{2h,k}\w_{h,\frac k3}}{\w_{2h,\frac k3}}e^{\frac{\pi i}{k}\left(-3\nu^2+\nu\right)h'}e^{\frac{2\pi i}{k}\left(-nh+\frac{mh'}{2}\right)}.
\end{align*}
For the range in the sums above, we may again choose any representatives $h,h' \pmod{k}$ satisfying $8\mid h'$. Note that the condition $8\mid h'$ enforces that we need to change $h'\mapsto h'+8k$. We also use the abbreviation
\begin{equation}\label{E:K6}
	K_k^{[6]}(\nu;n) := K_k^{[6]}(\nu;n,0).
\end{equation}

We also need the incomplete versions, $N+1\le\ell\le N+k-1$,
\begin{align*}
	\K_{k,\ell}^{[5]}(n,m) &:= \sum_{\substack{h\Pmod k^*\\N<k+k_1\le\ell\\8\mid h'}} \frac{\w_{h,k}\w_{2h,k}\w_{h,\frac k3}}{\w_{2h,\frac k3}}e^\frac{3\pi ih'}{4k}e^{\frac{2\pi i}{k}\left(-nh+\frac{mh'}{2}\right)},\\
	\K_{k,\ell}^{[6]}(\nu;n,m) &:= \sum_{\substack{h\Pmod k^*\\N<k+k_1\le\ell\\8\mid h'}} \frac{\w_{h,k}\w_{2h,k}\w_{h,\frac k3}}{\w_{2h,\frac k3}}e^{\frac{\pi i}{k}\left(-3\nu^2+\nu\right)h'}e^{\frac{2\pi i}{k}\left(-nh+\frac{mh'}{2}\right)}.
\end{align*}

Next we evaluate the multiplier.

\begin{lemma}\label{L:omega1}
	We have
	\[
		\frac{\w_{h,k}\w_{2h,k}\w_{h,\frac k3}}{\w_{2h,\frac k3}} = (-1)^\frac{k+1}{2}e^{\frac{4\pi ikh}{9}-\frac{2\pi i}{24k}\left(k^2-3\right)h'}.
	\]
\end{lemma}

\begin{proof}
	Now we use \eqref{E:omega} for $k$ odd to give
	\[
		\frac{\w_{h,k}\w_{2h,k}\w_{h,\frac k3}}{\w_{2h,\frac k3}} = -e^{2\pi iC},
	\]
	where
	\[
		C := \frac{1}{72k}\left(\left(-8h^2h'+5h'-16h-18\right)k^2+18k-9h'\right).
	\]
	Since $3\nmid h$, we have $h^2\equiv1\Pmod3$. Therefore
	\[
		-\frac{h^2h'}{3}\frac k3 \equiv -\frac{h'k}{9}\Pmod1.
	\]
	Thus
	\[
		C \equiv \frac{1-k}{4} + \frac{2k}{9}h + \left(-\frac k9+\frac{5k}{72}-\frac{1}{8k}\right)h'\Pmod1.
	\]
	This gives the claim.
\end{proof}

Similar to before, we obtain the following bounds for the Kloosterman sums.

\begin{lemma}\label{L:Kk56}
	We have, for $\e>0$,
	\[
		K_k^{[5]}(n,m),\ \K_{k,\ell}^{[5]}(n,m),\ \K_k^{[6]}(\nu;n,m),\ \K_{k,\ell}^{[6]}(\nu;n,m) \ll_\e n^\frac13k^{\frac23+\e}.
	\]
\end{lemma}

\subsection{$\gcd(k,6)=1$}

The Kloosterman sums are
\begin{align*}
	K_k^{[7]}(n,m) &:= \sum_{\substack{0\le h<k\\\gcd(h,k)=1\\24\mid h'}} \frac{\w_{h,k}\w_{2h,k}\w_{3h,k}}{\w_{6h,k}}e^\frac{3\pi ih'}{4k}e^{\frac{2\pi i}{k}\left(-nh+\frac{h'm}{6}\right)},\\
	K_k^{[8]}(\nu;n,m) &:= \sum_{\substack{0\le h<k\\\gcd(h,k)=1\\24\mid h'}} \frac{\w_{h,k}\w_{2h,k}\w_{3h,k}}{\w_{6h,k}}e^{\frac{\pi i}{k}\left(-3\nu^2-\nu\right)h'}e^{\frac{2\pi i}{k}\left(-nh+\frac{h'm}{6}\right)}.
\end{align*}
For the range in the sums above, we may again choose any representatives $h,h' \pmod{k}$ satisfying $24\mid h'$. Note that the condition $3\mid h'$ enforces that we need to change $h'\mapsto h'+3k$. We also use the abbreviation
\begin{equation}\label{E:K8}
	K_k^{[8]}(\nu;n) = K_k(\nu;n,0).
\end{equation}
We also need the incomplete versions, $N+1\le\ell\le N+k-1$,
\begin{align*}
	\K_{k,\ell}^{[7]}(n,m) &:= \sum_{\substack{h\Pmod k^*\\N<k+k_1\le\ell\\24\mid h'}} \frac{\w_{h,k}\w_{2h,k}\w_{3h,k}}{\w_{6h,k}}e^\frac{3\pi ih'}{4k}e^{\frac{2\pi i}{k}\left(-nh+\frac{h'm}{6}\right)},\\
	\K_{k,\ell}^{[8]}(\nu;n,m) &:= \sum_{\substack{h\Pmod k^*\\N<k+k_1\le\ell\\24\mid h'}} \frac{\w_{h,k}\w_{2h,k}\w_{3h,k}}{\w_{6h,k}}e^{\frac{\pi i}{k}\left(-3\nu^2+\nu\right)h'}e^{\frac{2\pi i}{k}\left(-nh+\frac{h'm}{6}\right)}.
\end{align*}

Again, we evaluate the multiplier.

\begin{lemma}\label{L:omega2}
	We have
	\[
		\frac{\w_{h,k}\w_{2h,k}\w_{3h,k}}{\w_{6h,k}} = (-1)^\frac{k+1}{2}e^\frac{10\pi i\left(k^2-1\right)h'}{72k}.
	\]
\end{lemma}

\begin{proof}
	We have, again using \eqref{E:omega} for $k$ odd,
	\[
		\frac{\w_{h,k}\w_{2h,k}\w_{3h,k}}{\w_{6h,k}} = -e^{2\pi iD},
	\]
	where
	\[
		D := \frac{1}{72k}\left(\left(5h'-18\right)k^2+18k-5h'\right) = \frac{5\left(k^2-1\right)}{72k}h' + \frac{1-k}{4}. \qedhere
	\]
\end{proof}

We now bound the Kloosterman sum.

\begin{lemma}\label{L:Kk78}
	We have, for $\e>0$,
	\[
		K_k^{[7]}(n,m),\ \K_{k,\ell}^{[7]}(n,m),\ K_k^{[8]}(\nu;n,m),\ \K_{k,\ell}^{[8]}(\nu;n,m) \ll_\e n^\frac13k^{\frac23+\e}.
	\]
\end{lemma}

\begin{proof}
	By \Cref{L:omega2}, we have
	\begin{align*}
		K_k^{[7]}(n,m) &= (-1)^\frac{k+1}{2}\sum_{\substack{h\Pmod k^*\\24\mid h'}} e^{\frac{2\pi i}{k}\left(-nh+\left(\frac{5k^2}{12}+\frac{11}{6}+m\right)\frac{h'}{6}\right)}\\
		&= (-1)^\frac{k+1}{2}\sum_{h\Pmod k^*} e^{\frac{2\pi i}{k}\left(-n[24]_kh+\left(\frac{5k^2}{12}+\frac{11}{6}+m\right)4h'\right)}.
	\end{align*}
	Now the bound follows as before, using \Cref{L:Klooster}. The remaining cases are treated in the same way.
\end{proof}

For $\gcd(k,6)=1$, we also require the following Kloosterman sums
\begin{equation}\label{E:K}
	\CK_k(n) := \sum_{h\Pmod k^*} \frac{\w_{h,k}\w_{2h,k}\w_{6h,k}}{\w_{3h,k}^3}e^{-\frac{2\pi inh}{k}}.
\end{equation}

\section{The Circle Method}\label{S:CM}

We follow Rademacher's approach \cite{Rad}. Note that $g_2$ is (up to a $q$-power) a modular form. Thus we can use \cite{RadZ} to obtain an exact formula for its coefficients. This yields the third term in \Cref{T:p2} as well a contribution if $\gcd(k,6)=2$ which turns out to cancel with the contribution from $g_1(q)$ (see the discussion in Subsubsection \ref{SSS:S21}).

From (4.1) of \cite{BM}, we have, for any\footnote{Note that in \cite{BM} we choose $N=\flo{\sqrt n}$ and in this paper, we let $N\to\infty$.} $N\in\N$,
\[
	a(n) = \sum_{\substack{0\le h<k\le N\\\gcd(h,k)=1}} e^{-\frac{2\pi inh}{k}}\int_{-\vth_{h,k}'}^{\vth_{h,k}''} g_1\left(e^{\frac{2\pi i}{k}(h+iz)}\right)e^\frac{2\pi nz}{k} d\Phi,
\]
where we use the notation from \Cref{S:pre}. We split
\[
	a(n) = \ssum_6 + \ssum_3 + \ssum_2 + \ssum_1,
\]
where $\sum_d$ denotes the sum over all terms $0\le h<k\le N$ with $\gcd(h,k)=1$ and $\gcd(k,6)=d$. In the following we repeatedly require the splitting
\begin{equation}\label{E:spit}
	\int_{-\vth_{h,k}'}^{\vth_{h,k}''} = \int_{-\frac{1}{k(k+N)}}^\frac{1}{k(k+N)} + \int_{-\frac{1}{k(k+k_1)}}^{-\frac{1}{k(k+N)}} + \int_\frac{1}{k(k+N)}^\frac{1}{k(k+k_2)}.
\end{equation}
We also use the decomposition
\[
	\int_{-\frac{1}{k(k+k_1)}}^{-\frac{1}{k(k+N)}} = \sum_{\ell=k+k_1}^{k+N-1} \int_{-\frac{1}{k\ell}}^{-\frac{1}{k(\ell+1)}}.
\]
This gives
\begin{equation}\label{E:split2}
	\sum_{\substack{0\le h<k\le N\\\gcd(h,k)=1\\\gcd(k,6)=d}} \int_{-\frac{1}{k(k+k_1)}}^{-\frac{1}{k(k+N)}} = \sum_{\substack{0\le h<k\le N\\\gcd(h,k)=1\\\gcd(k,6)=d}} \sum_{\ell=k+k_1}^{k+N-1} \int_{-\frac{1}{k\ell}}^{-\frac{1}{k(\ell+1)}} = \sum_{\substack{1\le k\le N\\\gcd(k,6)=d}} \sum_{\ell=N+1}^{k+N-1} \sum_{\substack{0\le h<k\\\gcd(h,k)=1\\N<k+k_1\le\ell}} \int_{-\frac{1}{k\ell}}^{-\frac{1}{k(\ell+1)}}.
\end{equation}

We have a similar splitting for $ \int_\frac{1}{k(k+N)}^\frac{1}{k(k+k_2)}$. Note that this splitting was not necessary in \cite{BM} as only rougher bounds for the Kloostermann sums were required for the asymptotic expansion.

\subsection{$6\mid k$}

We have,\footnote{Note that in \cite{BM} $S_{62}$ is called $S_{63}$.} by (4.2) of \cite{BM},
\[
	\ssum_6 = S_{61} + S_{62},
\]
where
\begin{align*}
	S_{61} &:= \sum_{\substack{0\le h<k\le N\\\gcd(h,k)=1\\6\mid k}} \frac{\w_{h,k}\w_{h,\frac k2}\w_{h,\frac k3}}{\w_{h,\frac k6}}(-1)^{\frac k2+1}e^{\frac{\pi i}{2}\left(1-\frac{3k}{2}\right)h'-\frac{2\pi inh}{k}}\int_{-\vth_{h,k}'}^{\vth_{h,k}''} e^\frac{2\pi nz}{k}g_1(q_1) d\Phi,\\
	S_{62} &:= \frac12\sum_{\substack{0\le h<k\le N\\\gcd(h,k)=1\\6\mid k}} \frac{\w_{h,k}\w_{h,\frac k2}\w_{h,\frac k3}}{k\w_{h,\frac k6}}e^{-\frac{2\pi inh}{k}}\sum_{\nu\Pmod k} (-1)^\nu e^{\frac{\pi i}{k}\left(-3\nu^2+\nu\right)h'}\\
	&\hspace{8cm}\times \int_{-\vth_{h,k}'}^{\vth_{h,k}''} e^\frac{2\pi nz}{k}e^{-\frac{\pi}{12kz}}\xi(q_1)I_{k,\nu}(z) d\Phi.
\end{align*}

\subsubsection{$S_{61}$}

We first investigate $S_{61}$, In order to use bounds for Kloosterman sums, we employ the splitting \eqref{E:spit}. We denote the corresponding terms $S_{61}^{[1]}$, $S_{61}^{[2]}$, and $S_{61}^{[3]}$, respectively.\\
4.1.1.1. $S_{61}^{[1]}$. We have
\[
	S_{61}^{[1]} = \sum_{\substack{1\le k\le N\\6\mid k}} (-1)^{\frac k2+1}\sum_{m\ge0} a(m)K_k^{[1]}(n,m)\int_{-\frac{1}{k(k+N)}}^\frac{1}{k(k+N)} e^{\frac{2\pi}{k}\left(nz-\frac mz\right)} d\Phi.
\]
Using \Cref{L:Kkmn} and the facts that $\re(z)=\frac{k}{N^2},\re(\frac1z)\ge\frac k2$, we have $S_{61}^{[1]}\to0$ as $N\to\infty$.\\
4.1.1.2. $S_{61}^{[2]}$ and $S_{61}^{[3]}$. The contributions $S_{61}^{[2]}$ and $S_{61}^{[3]}$ are treated in exactly the same way, thus we only need to consider $S_{61}^{[2]}$. Using \eqref{E:split2}, we obtain
\[
	S_{61}^{[2]} = \sum_{\substack{1\le k\le N\\6\mid k}} (-1)^{\frac k2+1}\sum_{m\ge0} a(m)\sum_{\ell=N+1}^{k+N-1} \K_{k,\ell}^{[1]}(n,m)\int_{-\frac{1}{k\ell}}^{-\frac{1}{k(\ell+1)}} e^{\frac{2\pi}{k}\left(nz-\frac mz\right)} d\Phi.
\]
Bounding this as before, we obtain that $S_{61}^{[2]}\to0$ as $N\to\infty$.

\subsubsection{$S_{62}$}

We use again the splitting \eqref{E:spit}. We denote the corresponding terms $S_{62}^{[1]}$, $S_{62}^{[2]}$, and $S_{62}^{[3]}$, respectively.\\
4.1.2.1. $S_{62}^{[1]}$. We have
\[
	S_{62}^{[1]} = \frac12\sum_{\substack{1\le k\le N\\6\mid k}} \frac1k\sum_{m\ge0} r(m)\sum_{\nu\Pmod k} (-1)^\nu K_k^{[2]}(\nu;n,m)\int_{-\frac{1}{k(k+N)}}^\frac{1}{k(k+N)} e^{\frac{2\pi}{k}\left(nz-\frac mz\right)}e^{-\frac{\pi}{12kz}}zI_{k,\nu}(z) d\Phi.
\]

Using \Cref{L:Kkmn} and \Cref{L:Jb} \ref{I:Jb1}, we have, as $N\to\infty$,
\begin{align*}
	S_{62}^{[1]} &\ll \sum_{\substack{1\le k\le N\\6\mid k}} \frac1ke^{\frac{2\pi n}{k}\frac{k}{N^2}}\sum_{m\ge0} |r(m)|e^{-\pi m}\sum_{\nu=1}^k n^\frac13k^{\frac23+\e}\int_{-\frac{1}{k(k+N)}}^\frac{1}{k(k+N)} \left|\CJ_{-\frac{1}{12},k,\nu}(z)\right| d\Phi\\
	&\ll n^\frac13\sum_{k=1}^N k^{-\frac13+\e}\sum_{\nu=1}^k \frac{1}{\left|\frac\pi2-\frac\pi k\left(\nu-\frac16\right)\right|}\frac{1}{k(k+N)} \ll \frac{n^\frac13}{N}\sum_{k=1}^N k^{-\frac13+\e}\log(k)\\
	&\ll n^\frac13N^{-\frac13+\e}\log(N) \to 0.
\end{align*}
4.1.2.2. $S_{62}^{[2]}$ and $S_{62}^{[3]}$. The contributions $S_{62}^{[2]}$ and $S_{62}^{[3]}$ are treated in exactly the same way, thus we only consider $S_{62}^{[2]}$. We use \eqref{E:split2} and obtain
\begin{multline*}
	S_{62}^{[2]} = \frac12\sum_{\substack{1\le k\le N\\6\mid k}} \frac1k\sum_{\nu\Pmod k} (-1)^\nu\sum_{m\ge0} r(m)\sum_{\ell=N+1}^{k+N-1} \K_{k,\ell}^{[2]}(\nu;n,m)\\
	\times \int_{-\frac{1}{k\ell}}^{-\frac{1}{k(\ell+1)}} e^{\frac{2\pi}{k}\left(nz-\frac mz\right)}\CJ_{-\frac{1}{12},k,\nu}(z)d\Phi.
\end{multline*}
As before, this vanishes as $N\to\infty$. Combining yields that $\sum_6\to0$ as $N\to\infty$.

\subsection{$\gcd(k,6)=2$}

By Subsection 4.3 of \cite{BM}, we have
\[
	\ssum_2 = S_{21} + S_{22},
\]
where 
\begin{align*}
	S_{21} &:= \frac14\sum_{\substack{0\le h<k\le N\\\gcd(h,k)=1\\\gcd(k,6)=2\\3\mid h'}} \frac{\w_{h,k}\w_{h,\frac k2}\w_{3h,k}}{\w_{3h,\frac k2}}(-1)^{\frac k2+1}e^{\frac{\pi i}{2}\left(1-\frac{3k}{2}\right)h'-\frac{2\pi inh}{k}}\\
	&\hspace{7cm}\times \int_{-\vth_{h,k}'}^{\vth_{h,k}''} e^{\frac{2\pi nz}{k}+\frac{2\pi}{9kz}}\frac{P\left(q_1^2\right)P\left(q_1^\frac13\right)}{P\left(q_1^\frac23\right)}f(q_1) d\Phi,\\
	S_{22} &:= \frac12\sum_{\substack{0\le h<k\le N\\\gcd(h,k)=1\\\gcd(k,6)=2\\3\mid h'}} \frac1k\frac{\w_{h,k}\w_{h,\frac k2}\w_{3h,k}}{\w_{3h,\frac k2}}e^{-\frac{2\pi inh}{k}}\sum_{\nu\Pmod k} (-1)^\nu e^{\frac{\pi ih'}{k}\left(-3\nu^2+\nu\right)}\\
	&\hspace{6.5cm}\times \int_{-\vth_{h,k}'}^{\vth_{h,k}''} ze^\frac{2\pi nz}{k}e^\frac{5\pi}{36kz}\frac{P\left(q_1^2\right)P\left(q_1^\frac13\right)}{P\left(q_1^\frac23\right)} I_{k,\nu}(z) d\Phi.
\end{align*}

\subsubsection{$S_{21}$}\label{SSS:S21}

In this case we have a principal part. The non-principal part is bounded exactly as for $S_{61}$ and vanishes as $N\to\infty$. We are left with
\[
	\CS_{21} := \frac14\sum_{\substack{0\le h<k\le N\\\gcd(h,k)=1\\\gcd(k,6)=2\\3\mid h'}} (-1)^{\frac k2+1}\frac{\w_{h,k}\w_{h,\frac k2}\w_{3h,k}}{\w_{3h,\frac k2}}e^{\frac{\pi i}{2}\left(1-\frac{3k}{2}\right)h'-\frac{2\pi inh}{k}}\int_{-\vth_{h,k}'}^{\vth_{h,k}''} e^{\frac{2\pi nz}{k}+\frac{2\pi}{9kz}} d\Phi.
\]
As shown in \cite{BM}, this exactly equals the negative of the modular piece and cancels the contribution of $g_2$ in that cusp. As alluded to above these two terms cancel.

\subsubsection{$S_{22}$}

We again have a principal part. The non-principal part is bounded exactly as for $S_{62}$ and vanishes as $N\to\infty$. We are left with
\begin{equation*}
	\CS_{22} := \frac12\sum_{\substack{0\le h<k\le N\\\gcd(h,k)=1\\\gcd(k,6)=2\\3\mid h'}} \frac1k\frac{\w_{h,k}\w_{h,\frac k2}\w_{3h,k}}{\w_{3h,\frac k2}}e^{-\frac{2\pi inh}{k}}\sum_{\nu\Pmod k} (-1)^\nu e^{\frac{\pi i}{k}\left(-3\nu^2+\nu\right)h'}\int_{-\vth_{h,k}'}^{\vth_{h,k}''}e^\frac{2\pi nz}{k}\CJ_{\frac{5}{36},k,\nu}(z)d\Phi.
\end{equation*}

We now use \Cref{L:Jb} \ref{I:Jb2}. The contribution from $\CJ_{\frac{5}{36},k,\nu}$ is bounded as for $S_{62}$ and vanishes as $N\to\infty$. Now for $\CJ_{\frac{5}{36},k,\nu}^*$, we again use \eqref{E:spit}. We denote the corresponding contribution by $\CS_{22}^{[1]}$, $\CS_{22}^{[2]}$, and $\CS_{22}^{[3]}$, respectively.\\
4.2.2.1. $S_{22}^{[1]}$. We have
\begin{multline*}
	\CS_{22}^{[1]} = \frac{\sqrt5\pi}{6\sqrt3}\sum_{\substack{1\le k\le N\\\gcd(k,6)=2}} \frac1k\sum_{\nu\Pmod k} (-1)^\nu K_k^{[4]}(\nu;n)\int_{-1}^1 \frac{L_k\left(n,\frac{5}{72}\left(1-x^2\right)\right)}{\cosh\left(\frac{\pi i}{k}\left(\nu-\frac16\right)-\frac{\pi\sqrt5x}{6\sqrt3k}\right)} dx\\
	+ \frac{\sqrt5i}{12\sqrt3}\sum_{\substack{1\le k\le N\\\gcd(k,6)=2}} \frac1k\sum_{\nu\Pmod k} (-1)^\nu K_k^{[4]}(\nu;n)\int_{-1}^1 \frac{1}{\cosh\left(\frac{\pi i}{k}\left(\nu-\frac16\right)-\frac{\pi\sqrt5x}{6\sqrt3k}\right)}\\
	\times \left(\CE_k^{[1]}\left(n,\frac{5}{72}\left(1-x^2\right)\right) + \CE_k^{[2]}\left(n,\frac{5}{72}\left(1-x^2\right)\right) + \CE_k^{[3]}\left(n,\frac{5}{72}\left(1-x^2\right)\right)\right) dx,
\end{multline*}
where ($w=\frac zk$), and
\begin{align*}
	L_k(n,y) &:= \frac{1}{2\pi i}\int_R e^{2\pi nw+\frac{2\pi y}{k^2w}} dw,\qquad \CE_k^{[1]}(n,y) := \int_{\frac{1}{N^2}+\frac{i}{k(k+N)}}^{-\frac{1}{N^2}+\frac{i}{k(k+N)}} e^{2\pi nw+\frac{2\pi y}{k^2w}} dw,\\
	\CE_k^{[2]}(n,y) &:= \int_{-\frac{1}{N^2}+\frac{i}{k(k+N)}}^{-\frac{1}{N^2}-\frac{i}{k(k+N)}} e^{2\pi nw+\frac{2\pi y}{k^2w}} dw,\qquad \CE_k^{[3]}(n,y) := \int_{-\frac{1}{N^2}-\frac{i}{k(k+N)}}^{\frac{1}{N^2}-\frac{i}{k(k+N)}} e^{2\pi nw+\frac{2\pi y}{k^2w}} dw,
\end{align*}
with $R$ the rectangle with edges $\pm\frac{1}{N^2}\pm\frac{i}{k(k+N)}$ surrounding $0$ counterclockwise. We first bound $\CE_k^{[1]}$ and $\CE_k^{[3]}$. On these ranges of integration, we have that (see \cite{Rad})
\[
	w = u\pm\frac{i}{k(k+N)},\qquad -\frac{1}{N^2} \le u \le \frac{1}{N^2},\qquad \re(w) = u \le \frac{1}{N^2},\qquad \re\left(\frac1w\right) \le 4k^2.
\]
Thus
\[
	\left|\CE_k^{[1]}\left(n,\frac{5}{72}\left(1-x^2\right)\right)\right|,\ \left|\CE_k^{[3]}\left(n,\frac{5}{72}\left(1-x^2\right)\right)\right| \le \frac{2}{N^2}e^{\frac{5\pi}{9}\left(1-x^2\right)+\frac{2\pi n}{N^2}}.
\]

For $\CE_k^{[2]}$ we have, again from \cite{Rad},
\[
	w = -\frac{1}{N^2} + iv,\qquad -\frac{1}{k(k+N)} \le v \le \frac{1}{k(k+N)},\qquad \re(w),\ \re\left(\frac1w\right) < 0.
\]
Thus
\begin{equation*}
	\CE_k^{[2]}\left(n,\frac{5}{144}\left(1-x^2\right)\right)<\frac{2}{kN}.
\end{equation*}
Thus $\CE_k^{[1]}$, $\CE_k^{[2]}$, and $\CE_k^{[3]}$ contribute, using \Cref{L:Kkl34},
\[
	\ll e^\frac{2\pi n}{N^2}\sum_{k=1}^N \frac1k\sum_{\nu=1}^k k^{\frac23+\e}n^\frac13\int_{-1}^1 \frac{1}{\left|\cosh\left(\frac{\pi i}{k}\left(\nu-\frac16\right)-\frac{\pi\sqrt5x}{6\sqrt3k}\right)\right|}\frac{1}{kN} e^{\frac{5\pi}{9}\left(1-x^2\right)} dx.
\]
We have for $\a\ge0$ and $0<\b<\pi$
\begin{equation}\label{E:cosbound}
	|\cosh(\a+i\b)| \ge \left|\sin\left(\frac\pi2-\b\right)\right| \gg \left|\frac\pi2-\b\right|.
\end{equation}
Note that $\nu-\frac16\ge0$. Thus the above is
\begin{align*}
	&\ll \frac{n^\frac13}{N}\sum_{k=1}^N k^{-\frac43+\e}\sum_{\nu=1}^k \frac{1}{\left|\frac\pi2-\frac\pi k\left(\nu-\frac16\right)\right|}\int_{-1}^1 e^{\frac{5\pi}{9}\left(1-x^2\right)} dx \ll \frac{n^\frac13}{N}\sum_{k=1}^N k^{-\frac43+\e}k\log(k)\\
	&\ll \frac{n^\frac13}{N}N^{\frac23+\e}\log(N) = n^\frac13N^{-\frac13+\e}\log(N) \to 0
\end{align*}
as $N\to\infty$. Thus the contributions of $\CE_k^{[1]}$, $\CE_k^{[2]}$, and $\CE_k^{[3]}$ vanish as $N\to\infty$.

Next, using the representation
\[
	I_\ell(x) := \sum_{m\ge0} \frac{1}{m!\Ga(m+\ell+1)}\left(\frac x2\right)^{2m+\ell},
\]
we evaluate
\begin{equation*}
	L_k(n,y) = \frac1k\sqrt\frac ynI_1\left(\frac{4\pi\sqrt{ny}}{k}\right),
\end{equation*}
plugging in the series for the exponential function. Thus we obtain overall, letting $N\to\infty$,
\[
	\CS_{22}^{[1]} = \frac{5\pi}{36\sqrt{6n}}\sum_{\substack{k\ge1\\\gcd(k,6)=2}} \frac{1}{k^2}\sum_{\nu\Pmod k} (-1)^\nu K_k^{[4]}(\nu;n)\int_{-1}^1 \frac{\sqrt{1-x^2}I_1\left(\frac{\pi\sqrt{10n\left(1-x^2\right)}}{3k}\right)}{\cosh\left(\frac{\pi i}{k}\left(\nu-\frac16\right)-\frac{\pi\sqrt5x}{6\sqrt3k}\right)} dx.
\]
This equals the first term in \Cref{T:p2}.\\
4.2.2.2. $S_{22}^{[2]}$ and $S_{22}^{[3]}$. We next turn to $S_{22}^{[2]}$; $S_{22}^{[3]}$ is bounded in exactly the same way. We have 
\begin{multline*}
	S_{22}^{[2]} = \frac12\sum_{\substack{1\le k\le N\\\gcd(k,6)=2}} \frac1k\sum_{\nu\Pmod k} (-1)^\nu\sum_{\ell=N+1}^{k+N-1} \K_{k,\ell}^{[4]}(\nu;n,0)\sqrt\frac{5}{36\cdot3}\\
	\times \int_{-1}^1 \frac{1}{\cosh\left(\frac{\pi i}{k}\left(\nu-\frac16\right)-\frac{\pi\sqrt\frac56x}{\sqrt3k}\right)}\int_{-\frac{1}{k\ell}}^{-\frac{1}{k(\ell+1)}} e^{2\pi nw+\frac{5\pi}{36k^2w}\left(1-x^2\right)} d\Phi dx.
\end{multline*}
We bound, following \cite{Rad},
\[
	\re\left(2\pi nw+\frac{5\pi}{36k^2w}\left(1-x^2\right)\right) \le \frac{2\pi n}{N^2} + \frac{5\pi}{9}\left(1-x^2\right).
\]
Thus, using \Cref{L:Kkl34} and \eqref{E:cosbound}, we obtain, for $N\to\infty$,
\begin{align*}
	S_{22}^{[2]} &\ll e^\frac{2\pi n}{N^2}\sum_{k=1}^N \frac1k\sum_{\nu=1}^k \sum_{\ell=N+1}^{k+N-1} k^{\frac23+\e}n^\frac13\int_{-1}^1 \frac{e^{\frac{5\pi}{9}\left(1-x^2\right)}}{\left|\cosh\left(\frac{\pi i}{k}\left(\nu-\frac16\right)-\frac{\pi\sqrt5x}{6\sqrt3k}\right)\right|} dx \int_{-\frac{1}{k\ell}}^{-\frac{1}{k(\ell+1)}} d\Phi\\
	&\ll n^\frac13\sum_{k=1}^N k^{-\frac13+\e}k\log(k)\frac{1}{k(N+k)} \ll \frac{n^\frac13}{N}\sum_{k=1}^N k^{-\frac13+\e}\log(k) \ll n^\frac13N^{-\frac13+\e}\log(N) \to 0.
\end{align*}

\subsection{$\gcd(k,6)=3$}

From Subsection 4.4 of \cite{BM}, we have
\[
	\ssum_3 = S_{31} + S_{32},
\]
where
\begin{align*}
	S_{31} &:= \frac12\sum_{\substack{0\le h<k\le N\\\gcd(h,k)=1\\\gcd(k,6)=3\\8\mid h'}} \frac{\w_{h,k}\w_{2h,k}\w_{h,\frac k3}}{\w_{2h,\frac k3}}(-1)^{\frac12(k-1)}e^{\frac{3\pi ih'}{4k}-\frac{2\pi inh}{k}}\\
	&\hspace{6.5cm}\times \int_{-\vth_{h,k}'}^{\vth_{h,k}''} e^{\frac{2\pi nz}{k}-\frac{\pi}{2kz}} \frac{P\left(q_1^\frac12\right)P\left(q_1^3\right)\w\left(q_1^\frac12\right)}{P\left(q_1^\frac32\right)} d\Phi,\\
	S_{32} &:= \frac12\sum_{\substack{0\le h<k\le N\\\gcd(h,k)=1\\\gcd(k,6)=3\\8\mid h'}} \frac{\w_{h,k}\w_{2h,k}\w_{h,\frac k3}}{k\w_{2h,\frac k3}}e^{-\frac{2\pi inh}{k}}\sum_{\nu\Pmod k} (-1)^\nu e^{\frac{\pi i}{k}\left(-3\nu^2+\nu\right)h'}\\
	&\hspace{6.5cm}\times \int_{-\vth_{h,k}'}^{\vth_{h,k}''} ze^{\frac{2\pi nz}{k}+\frac{\pi}{6kz}}\frac{P\left(q_1^\frac12\right)P\left(q_1^3\right)}{P\left(q_1^\frac32\right)}I_{k,\nu}(z) d\Phi,
\end{align*}
with the third order mock theta function
\[
	\w(q) := \sum_{n=0}^\infty \frac{q^{2n(n+1)}}{\left(q;q^2\right)_{n+1}^2}.
\]

\subsubsection{$S_{31}$}

In this case, we have no principal part. As for $S_{61}$ this contribution vanishes as $N\to\infty$.

\subsubsection{$S_{32}$}

Here, we have a principal part. The non-principal part is bounded exactly as for $S_{62}$ and vanishes as $N\to\infty$. We are left with
\begin{multline*}
	\CS_{32} := \frac12\sum_{\substack{0\le h<k\le N\\\gcd(h,k)=1\\\gcd(k,6)=3\\8\mid h'}} \frac1k\frac{\w_{h,k}\w_{2h,k}\w_{h,\frac k3}}{\w_{2h,\frac k3}}e^{-\frac{2\pi inh}{k}}\sum_{\nu\Pmod k} (-1)^\nu e^{\frac{\pi i}{k}\left(-3\nu^2+\nu\right)h'}\\
	\times \int_{-\vth_{h,k}'}^{\vth_{h,k}''} e^{\frac{2\pi nz}{k}}\CJ_{\frac16,k,\nu}(z) d\Phi.
\end{multline*}
Again we may change $\CJ_{\frac16,k,\nu}(z)$ into $\CJ_{\frac16,k,\nu}^*(z)$, using \Cref{L:Jb}. The error introduced vanishes for $N\to\infty$ as before.

Now for $\CJ_{\frac16,k,\nu}^*$, we again use \eqref{E:spit}. We denote the corresponding contributions by $\CS_{32}^{[1]}$, $\CS_{32}^{[2]}$, and $\CS_{32}^{[3]}$, respectively. First we have
\begin{multline*}
	\CS_{32}^{[1]} = \frac{\pi}{3\sqrt2}\sum_{\substack{1\le k\le N\\\gcd(k,6)=3}} \frac1k\sum_{\nu\Pmod k} (-1)^\nu K_k^{[6]}(\nu;n)\int_{-1}^1 \frac{L_k\left(n,\frac{1}{12}\left(1-x^2\right)\right)}{\cosh\left(\frac{\pi i}{k}\left(\nu-\frac16\right)-\frac{\pi x}{3\sqrt2k}\right)} dx\\
	+ \frac{i}{6\sqrt2}\sum_{\substack{1\le k\le N\\\gcd(k,6)=3}} \frac1k\sum_{\nu\Pmod k} (-1)^\nu K_k^{[6]}(\nu;n)\int_{-1}^{1}\frac{1}{\cosh\left(\frac{\pi i}{k}\left(\nu-\frac16\right)-\frac{\pi x}{3\sqrt2k}\right)}\\
	\times \left(\CE_k^{[1]}\left(n,\frac{1}{12}\left(1-x^2\right)\right) + \CE_k^{[2]}\left(n,\frac{1}{12}\left(1-x^2\right)\right) + \CE_k^{[3]}\left(n,\frac{1}{12}\left(1-x^2\right)\right)\right) dx.
\end{multline*}

As before, we show that the contributions from $\CE_k^{[1]}$, $\CE_k^{[2]}$, and $\CE_k^{[3]}$ vanish. Overall, we obtain, letting $N\to\infty$,
\[
	\CS_{32}^{[1]} = \frac{\pi}{6\sqrt{6n}}\sum_{\substack{k\ge1\\\gcd(k,6)=3}} \frac{1}{k^2}\sum_{\nu\Pmod k} (-1)^\nu K_k^{[6]}(\nu;n)\int_{-1}^1 \frac{\sqrt{1-x^2} I_1\left(\frac{2\pi\sqrt{n\left(1-x^2\right)}}{\sqrt3k}\right)}{\cosh\left(\frac{\pi i}{k}\left(\nu-\frac16\right)-\frac{\pi x}{3\sqrt2k}\right)} dx.
\]
This matches the second term in \Cref{T:p2}. Also as before, we show that $S_{32}^{[2]}$ and $S_{32}^{[3]}$ vanish, as $N\to\infty$.

\subsection{$\gcd(k,6)=1$}

By Subsection 4.5 of \cite{BM},
\[
	\ssum_1 = S_{11} + S_{12},
\]
where
\begin{align*}
	S_{11} &:= \frac12\sum_{\substack{0\le h<k\le N\\\gcd(h,k)=1\\\gcd(k,6)=1\\24\mid h'}} \frac{\w_{h,k}\w_{2h,k}\w_{3h,k}}{\w_{6h,k}}(-1)^{\frac12(k-1)}e^{\frac{3\pi ih'}{4k}-\frac{2\pi inh}{k}}\\
	&\hspace{6cm}\times \int_{-\vth_{h,k}'}^{\vth_{h,k}''} e^{\frac{2\pi nz}{k}-\frac{11\pi}{18kz}} \frac{P\left(q_1^\frac12\right)P\left(q_1^\frac13\right)\w\left(q_1^\frac12\right)}{P\left(q_1^\frac16\right)} d\Phi,\\
	S_{12} &:= \frac12\sum_{\substack{0\le h<k\le N\\\gcd(h,k)=1\\\gcd(k,6)=1\\24\mid h'}} \frac{\w_{h,k}\w_{2h,k}\w_{3h,k}}{k\w_{6h,k}}e^{-\frac{2\pi inh}{k}}\sum_{\nu\Pmod k} (-1)^\nu e^{\frac{\pi i}{k}\left(-3\nu^2+\nu\right)h'}\\
	&\hspace{6cm}\times \int_{-\vth_{h,k}'}^{\vth_{h,k}''} ze^{\frac{2\pi nz}{k}+\frac{\pi}{18kz}}\frac{P\left(q_1^\frac12\right)P\left(q_1^\frac13\right)}{P\left(q_1^\frac16\right)} I_{k,\nu}(z) d\Phi.
\end{align*}

\subsubsection{$S_{11}$}

Again, we have no principal part. As for $S_{61}$ this contribution vanishes as $N\to\infty$.

\subsubsection{$S_{12}$}

In this case, we have a principal part. The non-principal part is bounded exactly as for $S_{62}$ and vanishes, as $N\to\infty$. We are left with
\begin{multline*}
	\CS_{12} := \frac12\sum_{\substack{0\le h<k\le N\\\gcd(k,6)=1\\24\mid h'}} \frac1k\frac{\w_{h,k}\w_{2h,k}\w_{3h,k}}{\w_{6h,k}}e^{-\frac{2\pi inh}{k}}\sum_{\nu\Pmod k} (-1)^\nu e^{\frac{\pi i}{k}\left(-3\nu^2+\nu\right)h'}\\
\times \int_{-\vth_{h,k}'}^{\vth_{h,k}''} e^\frac{2\pi nz}{k}\CJ_{\frac{1}{18},k,\nu}(z) d\Phi.
\end{multline*}

Again we may change $\CJ_{\frac{1}{18},k,\nu}$ into $\CJ_{\frac{1}{18},k,\nu}^*(z)$, using \Cref{L:Jb}. The error introduced vanishes for $N\to\infty$. For $\CJ_{\frac{1}{18},k,\nu}^*$, we again employ \eqref{E:spit}. We denote the corresponding contributions by $\CS_{12}^{[1]}$, $\CS_{12}^{[2]}$, and $\CS_{12}^{[3]}$, respectively. First, we have 
\begin{multline*}
	\CS_{12}^{[1]} = \frac{\pi}{3\sqrt6}\sum_{\substack{1\le k\le N\\\gcd(k,6)=1}} \frac1k\sum_{\nu\Pmod k} (-1)^\nu K_k^{[8]}(\nu;n)\int_{-1}^1 \frac{L_k\left(n,\frac{1}{36}\left(1-x^2\right)\right)}{\cosh\left(\frac{\pi i}{k}\left(\nu-\frac16\right)-\frac{\pi x}{3\sqrt6k}\right)} dx\\
	+ \frac{i}{6\sqrt6}\sum_{\substack{1\le k\le N\\\gcd(k,6)=1}} \frac1k\sum_{\nu\Pmod k} (-1)^\nu K_k^{[8]}(\nu;n)\int_{-1}^1 \frac{1}{\cosh\left(\frac{\pi i}{k}\left(\nu-\frac16\right)-\frac{\pi x}{3\sqrt6k}\right)}\\
	\times \left(\CE_k^{[1]}\left(n,\frac{1}{36}\left(1-x^2\right)\right) + \CE_k^{[2]}\left(n,\frac{1}{36}\left(1-x^2\right)\right) + \CE_k^{[3]}\left(n,\frac{1}{36}\left(1-x^2\right)\right)\right) dx.
\end{multline*}
As before, we show that the contributions from $\CE_k^{[1]}$, $\CE_k^{[2]}$, and $\CE_k^{[3]}$ vanish. Thus overall
\[
	\CS_{12}^{[1]} = \frac{\pi}{18\sqrt{6n}}\sum_{\substack{k\ge1\\\gcd(k,6)=1}} \frac{1}{k^2}\sum_{\nu\Pmod k} (-1)^\nu K_k^{[8]}(\nu;n)\int_{-1}^1 \frac{\sqrt{1-x^2}I_1\left(\frac{2\pi\sqrt{n\left(1-x^2\right)}}{3k}\right)}{\cosh\left(\frac{\pi i}{k}\left(\nu-\frac16\right)-\frac{\pi x}{3\sqrt6k}\right)} dx.
\]
This matches the third term.

Again the contribution from $S_{12}^{[2]}$ and $S_{12}^{[3]}$ vanish. This completes the proof of Theorem \ref{T:p2}.

\end{document}